\documentclass{amsart}

\usepackage{amsmath,amssymb,amscd,verbatim}
\usepackage{amsthm}
\usepackage{enumerate}
\usepackage{url}
\usepackage[all]{xy}
\usepackage{wasysym}

\newtheorem{theorem}{Theorem}[section]
\newtheorem{fact}[theorem]{Fact}

\newtheorem{corollary}[theorem]{Corollary}
\newtheorem{proposition}[theorem]{Proposition}
\newtheorem{conjecture}[theorem]{Conjecture}

\theoremstyle{definition}
\newtheorem{definition}[theorem]{Definition}

\newtheorem{remark}[theorem]{Remark}
\newtheorem{question}[theorem]{Question}

\DeclareMathOperator{\tp}{tp}

\newcommand{\N}{\mathbb{N}}

\newcommand \Z{\mathbb{Z}}

\newcommand \BD{\operatorname{BD}}

\newcommand \NIP{\operatorname{NIP}}
\newcommand{\Erdos}{Erd\H{o}s}
\newcommand{\G}{\mathbb{G}}
\newcommand\seq{\subseteq}

\makeatletter
\def\indsym#1#2{%
  \setbox0=\hbox{$\m@th#1x$}%
  \kern\wd0%
  \hbox to 0pt{\hss$\m@th#1\mid$\hbox to 0pt{$\m@th#1^{#2}$}\hss}%
  \lower.9\ht0\hbox to 0pt{\hss$\m@th#1\smile$\hss}%
  \kern\wd0}
\newcommand{\ind}[1][]{\mathop{\mathpalette\indsym{#1}}}

\title{Definable sets containing productsets in expansions of groups}

\author[U. Andrews]{Uri Andrews}

\address{Department of Mathematics\\
University of Wisconsin\\
Madison, WI 53706\\
USA}
\email{andrews@math.wisc.edu}

\author[G. Conant]{Gabriel Conant}

\address{Department of Mathematics\\
University of Notre Dame\\
Notre Dame, IN, 46656\\
 USA}
\email{gconant@nd.edu}

\author[I. Goldbring]{Isaac Goldbring}
\thanks{Goldbring's work was partially supported by NSF CAREER grant DMS-1349399.}

\address{Department of Mathematics\\
University of California, Irvine\\
Irvine, CA 92697\\
 USA}
\email{isaac@math.uci.edu}

\begin{document}

\begin{abstract}
We consider the question of when sets definable in first-order expansions of groups contain the product of two infinite sets (we refer to this as the ``productset property").  We first show that the productset property holds for any definable subset $A$ of an expansion of a discrete amenable group such that $A$ has positive Banach density and the formula $x\cdot y\in A$ is stable. For arbitrary expansions of groups, we consider a ``$1$-sided" version of the productset property, which is characterized in various ways using coheir independence.  For stable groups, the productset property is equivalent to this $1$-sided version, and behaves as a notion of largeness for definable sets, which can be characterized by a natural weakening of model-theoretic genericity. Finally, we use recent work on regularity lemmas in distal theories to prove a definable version of the productset property for sets of positive Banach density definable in certain distal expansions of amenable groups.
\end{abstract}

\maketitle

\section{Introduction}

Given a subset $A$ of the set of integers $\Z$, the \emph{upper Banach density} of $A$ is 
\[
\BD(A)=\lim_{n\rightarrow\infty}\sup_{m\in\Z}\frac{|A\cap [m+1,m+n]|}{n},
\]
where, given integers $a<b$, $[a,b]$ denotes the interval $\{a,a+1,\ldots,b\}$. 

We are motivated by the following strengthening of a conjecture of \Erdos.

\begin{conjecture}\label{conj:Erdos}
Given $A\subseteq\Z$, if $\BD(A)>0$, then there are infinite sets $B,C\subseteq\Z$ such that $B+C\subseteq A$.
\end{conjecture}

This conjecture is stated for \emph{positive lower density} by \Erdos\ and Graham in \cite{ErGr}.  Progress was made recently by Di Nasso, Goldbring, Jin, Leth, Lupini, and Mahlburg \cite{DGJLLM}, where they showed that the conjecture holds if $\BD(A)>\frac{1}{2}$ and that a weak version of the conjecture holds in general:  if $\BD(A)>0$, then there are infinite sets $B,C\subseteq \Z$ and $k\in \Z$ such that $B+C\subseteq A\cup (A+k)$.  They also prove a variant of the aforementioned result for an arbitrary amenable group $G$,\footnote{Recall that a group is amenable if it admits a finitely additive left-invariant probabillity measure defined on all subsets of the group.  Using F{\o}lner sequences, one can make sense of the Banach density of a subset of an amenable group (see \cite{DGJLLM}).} where now one is interested in when sets in $G$ contain the product $B\cdot C$ of two infinite sets $B,C\subseteq G$.

In this note, we consider definable subsets of first-order expansions of groups which contain the product of two infinite sets (we call this the \emph{productset property}). Using the results in \cite{DGJLLM}, we show that if $G$ is an expansion of a discrete amenable group and $A\subseteq G$ is a definable set such that $\BD(A)>0$ and the formula $x\cdot y\in A$ is stable, then $A$ has the productset property. For the case $G=\Z$, we use nonstandard analysis to give a complete proof of this result, which does not directly rely on \cite{DGJLLM}. In particular, this confirms Conjecture \ref{conj:Erdos} for sets $A\subseteq \Z$ such that ``$x+y\in A$" is stable in the expansion $(\Z,+,0,A)$. 

We then analyze the model theoretic content of the productset property for definable sets in first-order expansions of groups. Specifically, we consider a more flexible notion, the \emph{1-sided productset property} (see Definition \ref{def:1side}), which coincides with the productset property in the case of stable groups. Motivated by an unpublished observation of DiNasso (Proposition \ref{prop:MDN}), we show that the productset property for a definable set $A$, in an arbitrary first-order expansion of a group $G$, is equivalent to the existence of nonalgebraic global types $p,q$ finitely satisfiable in $G$ such that the formula $x\cdot y\in A$ is contained in $p(x)\otimes q(y)$ and $q(y)\otimes p(x)$. Using this we conclude that, when $G$ is stable, the productset property for definable $A\subseteq G$ is equivalent to the existence of a nonstandard element $c\not\in G$ such that the translate $A\cdot c$ has infinitely many solutions in $G$. Thus the productset property can be viewed as a natural weakening of the notion of \emph{generic} definable sets (where this condition holds for \emph{all} such translates). We also consider a finitary version of the productset property and its connection with \emph{generically stable types}.  

Finally, we consider the finitary productset property in the setting of distal groups. Using recent work of Chernikov and Starchenko \cite{ChSt}, we show that for distal groups, the finitary productset property is always witnessed by a definable family of sets. Using this, we show that if $G$ is a distal expansion of an amenable group, and if $G$ eliminates the quantifier $\exists^\infty$, then any definable subset of $G$ with positive Banach density has the productset property witnessed by definable sets.

\section{\Erdos's conjecture in the stable setting}

Motivated by Conjecture \ref{conj:Erdos}, we define the following properties of subsets of groups.

\begin{definition}\label{def:1side}
Let $G$ be a group.
\begin{enumerate}
\item A set $A\subseteq G$ has the \emph{productset property} if there are infinite $B,C\subseteq G$ such that $B\cdot C\subseteq A$.
\item A set $A\subseteq G$ has the \emph{$1$-sided productset property} if there are infinite sequences $(b_i)_{i<\omega}$ and $(c_i)_{i<\omega}$ in $G$ such that $b_i\cdot c_j\in A$ for all $i\leq j<\omega$.
\end{enumerate}
\end{definition}

If the group $G$ is abelian, then we will speak of the (1-sided) sumset property rather than the (1-sided) productset property.

The next fact is a consequence of \cite[Lemma 3.4]{DGJLLM}, which is proved using nonstandard analysis and technical results from ergodic theory.

\begin{fact}\label{fact:DGJLLM}
Let $G$ be an amenable group. If $\BD(A)>0$, then $A$ has the $1$-sided productset property. 
\end{fact}

In \cite[Lemma 3.3]{DGJLLM}, a proof is given for $G=\Z$ that is simpler than the proof for a general amenable group.  That being said, \cite[Lemma 3.3]{DGJLLM} has a stronger conclusion than Fact \ref{fact:DGJLLM} and thus we take the opportunity here to give a short proof of Fact \ref{fact:DGJLLM} for $G=\Z$, suggested to us by Renling Jin.  The proof is given in terms of nonstandard analysis and we follow the usual terminology and notations that appear in the literature.  For example, $\N^*$ denotes a sufficiently saturated nonstandard extension of $\N$. 

\begin{proposition}\label{positiveBD}
Given $A\subset\Z$, if $\BD(A)>0$, then $A$ has the $1$-sided sumset property.
\end{proposition}
\begin{proof}
Without loss of generality, $A\subseteq \N$.  Take hyperfinite $I\subseteq \N^*\backslash \N$ such that $|A^*\cap I|/|I|\approx \alpha:=\BD(A)$.  We claim that it suffices to find $x\in I$ such that $(A^*-x)\cap \N$ is infinite.  Indeed, suppose $(b_i)_{i<\omega}$ is an infinite sequence in $\N$ such that $b_i+x\in A^*$ for all $i<\omega$. Then for any $j<\omega$, we may apply transfer to find $c_j\in \N$ such that $b_i+c_j\in A$ for all $i\leq j$.

Write $I=[a,b]$ and fix $N\in \N^*\backslash \N$ such that $N/(b-a)\approx 0$.  In order to find $x\in I$ such that $(A^*-x)\cap \N$ is infinite, it suffices to find $x\in I$ such that $|A^*\cap [x,x+n)|/n\geq \alpha/2$ for all $n\in \N^*$ with $n\leq N$.  Suppose, towards a contradiction, that no such $x\in I$ exists.  We define a hyperfinite sequence $(x_k)_{k\leq K}$ from $I$ as follows.  Set $x_0:=a$.  Suppose that $x_0,\ldots,x_k$ have been constructed such that $x_{i+1}-x_i\leq N$ and $|A^*\cap [x_i,x_{i+1})|/(x_{i+1}-x_i)<\alpha/2$ for all $i<k$.  If $b-x_k<N$, set $x_{k+1}:=b$ and terminate the construction.  Otherwise, set $x_{k+1}\in I$ to be such that $x_{k+1}-x_k\leq N$ and $|A^*\cap [x_k,x_{k+1})|<\alpha/2$.  It follows that 
$$|A^*\cap I|=\sum_{k=0}^{K-2}|A^*\cap [x_k,x_{k+1})|+|A^*\cap [x_{K-1},x_K]|<\frac{\alpha}{2}(x_{K-1}-a)+N,$$ whence
$$\frac{|A^*\cap I|}{|I|}<\frac{\alpha}{2}\cdot \frac{x_{K-1}-a}{b-a}+\frac{N}{b-a}\approx \frac{\alpha}{2},$$ contradicting the choice of $I$.  
\end{proof}

\begin{question}
Can the proof of Proposition \ref{positiveBD} be generalized to the case of an arbitrary amenable group?
\end{question}

Suppose that $G$ is a first-order structure expanding a group and $A\subseteq G$ is definable.  We abuse terminology and say that $A$ is \emph{stable} if the two-variable formula $x\cdot y\in A$ does not have the order property.

\begin{proposition}\label{prop:Ramsey}
Suppose that $G$ is a first-order structure expanding a group and that $A\subseteq G$ is a stable definable set.  Then $A$ has the productset property if and only if it has the $1$-sided productset property.
\end{proposition}
\begin{proof}
One direction is trivial. For the other direction, fix a stable definable set $A\subseteq G$ which has the $1$-sided productset property witnessed by infinite sequences $(b_i)_{i<\omega}$ and $(c_i)_{i<\omega}$ in $G$. Let $P_1=\{(i,j):i>j,~b_i\cdot c_j\in A\}$ and $P_2=\{(i,j):i>j,~b_i\cdot c_j\not\in A\}$. By Ramsey's theorem, there is an infinite set $I$ of indices and some $t\in\{1,2\}$ such that $(i,j)\in P_t$ for all $i,j\in I$ with $i>j$. If $t=2$, then $(b_i)_{i\in I}$ and $(c_i)_{i\in I}$ witness the order property for $x\cdot y\in A$, which is a contradiction. Therefore $t=1$, and so, setting $B=\{b_i:i\in I\}$ and $C=\{c_i:i\in I\}$, we have $B\cdot C\subseteq A$.
\end{proof}

Fact \ref{fact:DGJLLM} and Proposition \ref{prop:Ramsey} yield the productset property for stable definable sets of positive Banach density in expansions of amenable groups.

\begin{theorem}
Suppose that $G$ is a first-order structure expanding a discrete amenable group and that $A\subseteq G$ is a stable definable set.  If $\BD(A)>0$, then there are infinite $B,C\subseteq G$ such that $B\cdot C\subseteq A$.
\end{theorem}

\begin{remark}
In the model-theoretic setting, a first-order structure expanding a group $G$ is \emph{definably amenable} if there is a finitely additive left-invariant probability measure on the definable subsets of $G$. It is worth emphasizing that, in this paper, we do not consider this weaker notion of amenability.
\end{remark}

\section{Substantial subsets of groups}\label{sec:PSP}

\subsection{The productset property and coheir substantiality}\label{sec:ch}
 The original motivation for the present paper comes from the following unpublished observation of Mauro DiNasso.  We thank him for his permission in allowing us to include this result and its proof.

\begin{proposition}\label{prop:MDN}
$A\subseteq \Z$ has the sumset property if and only if there are nonprincipal ultrafilters $\mathcal{U}$ and $\mathcal{V}$ on $\Z$ such that $A\in (\mathcal{U}\oplus \mathcal{V})\cap (\mathcal{V}\oplus \mathcal{U})$.
\end{proposition}

Here, $\mathcal{U}\oplus \mathcal{V}$ is the ultrafilter on $\mathbb{Z}$ defined by setting $A\in \mathcal{U}\oplus \mathcal{V}$ if and only if $A-\mathcal{V}\in \mathcal{U}$, where $A-\mathcal{V}:=\{k\in \mathbb{Z} \ : \ A-k\in \mathcal{V}\}$.

\begin{proof}[Proof of Proposition \ref{prop:MDN}]
First suppose that $B,C\subseteq \mathbb{Z}$ are infinite and that $B+C\subseteq A$.  It follows that the family $\{B\}\cup \{A-c \ : \ c\in C\}$ has the finite intersection property, whence there is a nonprincipal ultrafilter $\mathcal{U}$ on $\mathbb{Z}$ extending this family.  In the same way, there is a nonprincipal ultrafilter $\mathcal{V}$ on $\mathbb{Z}$ extending the family $\{C\}\cup \{A-b \ : \ b\in B\}$.  Since $B\subseteq A-\mathcal{V}$ and $B\in \mathcal{U}$, we have that $A-\mathcal{V}\in \mathcal{U}$, that is, $A\in \mathcal{U}\oplus \mathcal{V}$; the argument that $A\in \mathcal{V}\oplus \mathcal{U}$ is identical.

Now suppose that $A\in (\mathcal{U}\oplus \mathcal{V})\cap (\mathcal{V}\oplus \mathcal{U})$ for nonprincipal ultrafilters $\mathcal{U}$ and $ \mathcal{V}$ on $\mathbb{Z}$.  Pick $b_0\in A-\mathcal{V}$ and $c_0\in A-\mathcal{U}$ arbitrarily.  Now pick $b_1\in (A-\mathcal{V})\cap (A-c_0)$, with $b_1\neq b_0$, and $c_1\in (A-\mathcal{U})\cap (A-b_0)\cap (A-b_1)$, with $c_1\neq c_0$.  Now pick $b_2\in (A-\mathcal{V})\cap (A-c_0)\cap (A-c_1)$, with $b_2\not\in\{b_0,b_1\}$, and then $c_2\in(A-\mathcal{U})\cap (A-b_0)\cap (A-b_1)\cap (A-b_2)$, with $c_2\not\in\{c_0,c_1\}$.  Continuing in this way, we get infinite sequences $(b_i)_{i<\omega}$ and $(c_j)_{j<\omega}$ such that $b_i+c_j\in A$ for all $i,j<\omega$.
\end{proof}

\begin{remark}\label{rem:MDN}
Proposition \ref{prop:MDN} also holds in an arbitrary abelian group (with virtually the same proof).
\end{remark}

For the rest of Section \ref{sec:PSP}, we let $G$ denote a fixed first-order expansion of a group. We also let $T$ denote the complete theory of $G$ and we let $\G$ be a sufficiently saturated monster model of $T$. We write $A\subset \G$ to mean $A$ is a ``small" subset, in the sense that $\G$ is $|A|^+$-saturated.  Unless otherwise specified, we use $\varphi(x)$ to denote a formula in the single variable $x$ with parameters from $G$. We say that a formula $\varphi(x)$ has the \emph{$1$-sided productset property} (resp. \emph{productset property}) if the set $\varphi(G)$ has the $1$-sided productset property (resp. productset property). 

\begin{remark}
In this model-theoretic context, when we say a formula $\varphi(x)$ has the ($1$-sided) productset property, it is worth emphasizing that we \emph{do not} require the witnessing sets $B$ and $C$ to be definable. For example, suppose $G$ is strongly minimal. Then any infinite definable subset of $G$ is cofinite, and thus has the productset property. On the other hand, the only subset of $G$ with the productset property witnessed by definable $B$ and $C$ is $G$ itself (since, in any group, the product of two cofinite sets is the whole group). In Section \ref{sec:distal}, we will consider definable versions of the productset property in the setting of \emph{distal} groups.
\end{remark}

Much of the work in this section is motivated by the consideration of Proposition \ref{prop:MDN} in light of certain analogies between types and ultrafilters. First, recall that there is a map $(\mathcal{U},\mathcal{V})\mapsto \mathcal{U}\otimes \mathcal{V}$ from pairs of ultrafilters on a set $X$ to ultrafilters on $X^2$ given by declaring, for $E\subseteq X^2$, that $E\in \mathcal{U}\otimes \mathcal{V}$ if and only if $\{x\in X  :  E_x\in \mathcal{V}\}\in \mathcal{U}$, where $E_x$ is the fiber of $E$ over $x$, that is, $E_x:=\{y\in X  :  (x,y)\in E\}$.  In the case $X=\mathbb{Z}$, the operation $\oplus$ above is thus the pushforward of the operation $\otimes$ under the map $(x,y)\mapsto x+y$. 

In the model-theoretic setting, there is an operation $\otimes$ on global types that is meant to mimic the operation $\otimes$ on ultrafilters. Specifically, fix $A\subset\G$ and suppose $p,q\in S_1(\G)$ are global types such that $p$ is $A$-invariant. Define the global type $p(x)\otimes q(y)$ so that, given $A\subseteq B\subset\G$ and a formula $\theta(x,y)$ with parameters in $B$, $\theta(x,y)\in p(x)\otimes q(y)$ if and only if $\theta(x,c)\in p$ for some (any) $c\models q|_B$. (See \cite[Section 2.2]{simon} for details.) For any $A\subset\G$, there is also a surjective map from ultrafilters on $A$ to global types in $S_1(\G)$ which are finitely satisfiable in $A$, given by $\mathcal{U}\mapsto \{\theta(x,c):\theta(A,c)\in\mathcal{U}\}$ (see \cite[Example 2.1.7]{simon}). Finally, note that if $p\in S_1(\G)$ is finitely satisfiable in some set $A\subset\G$, then $p$ is $A$-invariant. Altogether, this motivates the following model-theoretic interpretation of Proposition \ref{prop:MDN}.

\begin{theorem}\label{thm:PSP}
A formula $\varphi(x)$ has the productset property if and only if there are nonalgebraic global types $p,q\in S_1(\G)$ such that $p,q$ are finitely satisfiable in $G$ and $\varphi(x\cdot y)\in (p(x)\otimes q(y))\cap(q(y)\otimes p(x))$.
\end{theorem}
\begin{proof}
First, suppose $\varphi(x)$ has the productset property, and let $B,C\seq G$ be infinite with $B\cdot C\seq\varphi(G)$. Since $B$ and $C$ are infinite, the type $\{x\neq g:g\in G\}$ is finitely satisfiable in both $B$ and $C$, and thus extends to global types $p,q\in S_1(\G)$, which are finitely satisfiable in $B$ and $C$, respectively. Note then that $p,q$ are necessarily nonalgebraic. For a contradiction, suppose $\neg\varphi(x\cdot y)\in p(x)\otimes q(y)$. Let $c^*\models q|_G$. Then $\neg\varphi(x\cdot c^*)\in p$ so there is $b\in B$ such that $\neg\varphi(b\cdot c^*)$ holds. Then $\neg\varphi(b\cdot y)\in q$, and so there is $c\in C$ such that $\neg\varphi(b\cdot c)$ holds, which is a contradiction. By a similar argument, $\varphi(x\cdot y)\in q(y)\otimes p(x)$. 

Now suppose $p,q\in S_1(\G)$ satisfy the conditions of the theorem. In a sufficiently saturated extension of $\G$, let $(b^*,c^*)\models p(x)\otimes q(y)$ and let $(b',c')\models q(y)\otimes p(x)$. By assumption, $\varphi(b^*\cdot c^*)$ and $\varphi(b'\cdot c')$ hold. Let $p^*=p|\G c^*$ and $q'=q|\G b'$. Note that $b^*\models p^*$, $c'\models q'$, and $p^*,q'$ are nonalgebraic and finitely satisfiable in $G$.  Fix $n<\omega$ and suppose we have constructed sequences $(b_i)_{i<n}$ and $(c_j)_{j<n}$ such that $\varphi(b_i\cdot c_j)$ holds for all $i,j<n$, $\varphi(b_i\cdot c^*)$ for all $i<n$, and $\varphi(b^*\cdot c_j)$ holds for all $j<n$. The formula
\[
\bigwedge_{j<n}\varphi(x\cdot c_j)\wedge\varphi(x\cdot c^*)\wedge\bigwedge_{i<n}x\neq b_i
\]
is realized by $b^*$ and thus is in $p^*$. So we may find a realization $b_n$ in $G$. Since $c^*\equiv_G c'$, we have $\varphi(b_i\cdot c')$ for all $i\leq n$. So the formula
\[
\bigwedge_{i\leq n}\varphi(b_i\cdot y)\wedge\varphi(b'\cdot y)\wedge\bigwedge_{j<n}y\neq c_j
\]
is realized by $c'$ and thus is in $q'$. So we may find a realization $c_n$ in $G$. Since $b^*\equiv_G b'$, we have $\varphi(b^*\cdot c_n)$. This constructs infinite $B,C$ such that $B\cdot C\seq\varphi(G)$.
\end{proof}

In the next subsection, we will analyze the productset property in the context of model-theoretic genericity in stable and simple theories. Toward this end, a natural question is when, in the above characterization of the productset property, we can choose realizations of the types $p$ and $q$, which are mutually independent with respect to finite satisfiability. This motives the next definition.

\begin{definition}\label{def:cs}
A formula $\varphi(x)$ is \emph{$1$-sided coheir substantial} (resp. \emph{coheir substantial}) if there are $b,c\in \G\backslash G$ such that $\varphi(b\cdot c)$ holds and $\tp(b/Gc)$ is (resp. $\tp(b/Gc)$ and $\tp(c/Gb)$ are) finitely satisfiable in $G$.
\end{definition}

\begin{theorem}\label{thm:1ch}
Fix a formula $\varphi(x)$. The following are equivalent:
\begin{enumerate}[$(i)$]
\item $\varphi(x)$ has the $1$-sided productset property.
\item $\varphi(x)$ is $1$-sided coheir substantial.
\item There are nonalgebraic global types $p,q\in S_1(\G)$ such that $p$ is finitely satisfiable in $G$ and $\varphi(x\cdot y)\in p(x)\otimes q(y)$.
\end{enumerate}
\end{theorem}
\begin{proof}
$(i)\Rightarrow (ii)$: Assume that $\varphi(x)$ has the $1$-sided productset property as witnessed by the infinite sequences $(b_i)_{i<\omega}$ and $(c_i)_{i<\omega}$ from $G$. By saturation, we may find $c\in \G\backslash G$ such that $\varphi(b_i\cdot c)$ holds for all $i<\omega$. Then $\{\varphi(x\cdot c)\}\cup \{x\neq g:g\in G\}$ is finitely satisfiable in $G$, and thus extends to a complete type $p\in S_1(Gc)$ which is finitely satisfiable in $G$. If $b\in \G$ is a realization of $p$, then $b,c\in \G\backslash G$ witness that $\varphi(x)$ is $1$-sided coheir substantial. 

$(ii)\Rightarrow (iii)$: Assume that $\varphi(x)$ is coheir substantial as witnessed by $b,c\in \G\backslash G$. We may extend $\tp(b/Gc)$ to a global type $p\in S_1(\G)$, which is finitely satisfiable in $G$. Since $b\not\in G$, it follows that $p$ is nonalgebraic. Since $\tp(b/Gc)$ is finitely satisfiable in $G$, it follows that $c\not\in\text{acl}(Gb)$, and so we may extend $\tp(c/Gb)$ to a nonalgebraic global type $q\in S_1(\G)$. Since $c\models q|_G$ and $\varphi(x\cdot c)\in p$, we have $\varphi(x\cdot y)\in p(x)\otimes q(y)$, as desired. 

$(iii)\Rightarrow (i)$:  Assume that $\varphi(x)$ satisfies $(iii)$ as witnessed by $p,q\in S_1(\G)$. Let $c\in\G\backslash G$ realize $q|_G$. Since $\varphi(x\cdot y)\in p(x)\otimes q(y)$, we have $\varphi(x\cdot c)\in p$. Let $b\models p|_{Gc}$. Then $\varphi(b\cdot c)$ holds and $\tp(b/Gc)$ is finitely satisfiable in $G$ (since it is contained in $p$). We inductively construct infinite sequences $(b_i)_{i<\omega}$ and $(c_i)_{i<\omega}$ in $G$ such that $\varphi(b_i\cdot c_j)$ holds for all $i\leq j<\omega$. In particular, fix $n<\omega$ and assume we have distinct $(b_i)_{i<n}$ and distinct $(c_i)_{i<n}$ such that $\varphi(b_i\cdot c_j)$ holds for $i\leq j<n$ and $\varphi(b_i\cdot c)$ holds for $i<n$. The formula $\varphi(x\cdot c)\wedge\bigwedge_{i<n}x\neq b_i$ is in $\tp(b^*/Gc)$ and thus realized by some $b_n\in G$. Then $\G\models \bigwedge_{i\leq n}\varphi(b_i\cdot c)\wedge\bigwedge_{j<n}c\neq c_j$ and so, by elementarity, there is $c_n\in G\backslash(c_j)_{j<n}$ such that $\varphi(b_i\cdot c_n)$ holds for all $i\leq n$. 
\end{proof}

When $T$ is stable, the ``$2$-sided" version of the previous result holds.

\begin{theorem}\label{thm:ch}
Fix a formula $\varphi(x)$.
\begin{enumerate}[$(a)$]
\item If $\varphi(x)$ is coheir substantial then it has the productset property.
\item Assume $T$ is stable. Then $\varphi(x)$ is coheir substantial if and only if it has the productset property.
\end{enumerate}
\end{theorem}
\begin{proof}
Part $(a)$. Follow the proof of Theorem \ref{thm:1ch}[$(iii)\Rightarrow (i)$], with $q\in S_1(\G)$ chosen to be finitely satisfiable in $M$. We have $\varphi(x\cdot y)\in p(x)\otimes q(y)$ by the same argument. Since $b\models p|_G$ and $\varphi(b\cdot y)\in q$, we also have $\varphi(x,y)\in q(y)\otimes p(x)$.

Part $(b)$. We only need to show the reverse implication. Since $T$ is stable, coheir independence satisfies symmetry (over models), and so coheir substantial coincides with $1$-sided coheir substantial. Thus the result follows from Theorem \ref{thm:1ch}.
\end{proof}

Part $(b)$ of the previous result can fail in general (we thank Pierre Simon for this observation), as witnessed by the following example.

\begin{remark}\label{rem:counter}
Let $G=(\Z,+,<,0)$ and let $\varphi(x)$ be $x>0$ (which clearly has the sumset property). For a contradiction, suppose there are $b,c\in \G\backslash G$ such $b+c>0$ and $\tp(b/\Z c)$ and $\tp(c/\Z b)$ are both finitely satisfiable in $\Z$. Without loss of generality, suppose $b\leq c$. Then $b\leq y$ and $b+y>0$ are both realized in $\Z$, which is impossible. 
\end{remark}


\begin{remark}\label{rem:formula}
Theorems \ref{thm:PSP}, \ref{thm:1ch}, and \ref{thm:ch} can be directly adapted to outside of the group context, in which $G$ is replaced by an arbitrary first-order structure $M$, $\varphi(x\cdot y)$ is replaced by a formula $\varphi(x,y)$, and the \emph{productset property} is interpreted to mean $B\times C\seq\varphi(M)$ for some infinite $B,C\seq M$. Proposition \ref{prop:Ramsey} can also be generalized to this context to show that the productset property and $1$-sided productset property coincide for stable formulae.
\end{remark}

\subsection{Nonforking substantiality and generic formulae in simple theories}

A key part of the proof of Theorem \ref{thm:ch}$(b)$ is the symmetry of coheir independence in stable theories. In the more general context of simple theories, coheir independence can fail symmetry, but nonforking independence (which coincides with coheir independence in the stable case) remains symmetric. This motivates the following definition.

\begin{definition}
A formula $\varphi(x)$ is \emph{$1$-sided nonforking substantial} (resp. \emph{nonforking substantial}) if there are $b,c\in \G\backslash G$ such that $\varphi(b\cdot c)$ holds and $\tp(b/Gc)$ does (resp. $\tp(b/Gc)$ and $\tp(c/Gb)$ do) not fork over $G$.
\end{definition}

Recall that, for any theory $T$, coheir independence is stronger than nonforking independence, whence ($1$-sided) coheir substantial implies ($1$-sided) nonforking substantial. Applying Theorem \ref{thm:1ch}, we obtain the following conclusion.

\begin{corollary}\label{cor:1psp-fork}
If a formula $\varphi(x)$ has the $1$-sided productset property then $\varphi(x)$ is $1$-sided nonforking substantial.
\end{corollary}

Combined with Fact \ref{fact:DGJLLM}, we have:

\begin{corollary}\label{cor:BDnf}
Suppose $G$ is a first-order expansion of an amenable group and $A\subseteq G$ is definable. If $\BD(A)>0$, then $A$ is $1$-sided nonforking substantial.
\end{corollary}

The next corollary collects the stronger conclusions obtained in the case that $T$ is stable or simple. 

\begin{corollary}\label{cor:stable-simple}
Fix a formula $\varphi(x)$.
\begin{enumerate}[$(a)$]
\item Assume $T$ is stable. The following are equivalent.
\begin{enumerate}[$(i)$]
\item $\varphi(x)$ is ($1$-sided) nonforking substantial.
\item $\varphi(x)$ is ($1$-sided) coheir substantial.
\item $\varphi(x)$ has the ($1$-sided) productset property.
\end{enumerate}
\item Assume $T$ is simple. Then $\varphi(x)$ is nonforking substantial if and only if it is $1$-sided nonforking substantial.
\item Assume $T$ is simple. If $\varphi(x)$ has the productset property then it is nonforking substantial. 
\end{enumerate}
\end{corollary}

In the case that $T$ is stable, we have a single notion of ``substiantiality", which can be characterized in several interesting ways. Therefore, in the stable case,  we simply say that $\varphi(x)$ is \emph{substantial} if it satisfies one of the equivalent conditions in part $(b)$ of the previous corollary.

We again note that, since coheir independence implies nonforking independence in any theory, part $(c)$ of the previous corollary holds whenever $T$ satisfies the conclusion of Theorem \ref{thm:ch}$(b)$. 

\begin{question}
Does Corollary \ref{cor:stable-simple}$(c)$ hold without the assumption that $T$ is simple (or, under weaker assumptions such as NIP)?
\end{question}

\begin{question}
Is there a model theoretically ``tame" group $G$ (e.g. simple or $\NIP$) in which some formula $\varphi(x)$ is ($1$-sided) nonforking substantial but not ($1$-sided) coheir substantial?
\end{question}

 Our next goal is to connect the notion of substantial definable sets to the well-studied notion of model-theoretic genericity.

\begin{definition}\label{def:gen}$~$
\begin{enumerate}
\item A type $p\in S_1(G)$ is \emph{generic} if, for all $b,c\in \G$, if $b\models p$ and $\tp(b/Gc)$ does not fork over $G$, then $\tp(c\cdot b/Gc)$ does not fork over $G$.
\item A formula $\varphi(x)$ is \emph{generic} if it is contained in a generic type $p\in S_1(G)$.
\item A formula $\varphi(x)$ is \emph{syndetic} if $G$ is covered by finitely many left translates of $\varphi(G)$.
\end{enumerate}
\end{definition}

\begin{remark}
Recall that a set $A\subseteq\mathbb{N}$ of natural numbers is called \emph{syndetic} if finitely many translates of $A$ cover $\mathbb{N}$. In classical stable group theory, what we have called ``syndetic" is often used as the definition of ``generic". Indeed, for a stable group $G$, a formula $\varphi(x)$ is generic if and only if $G$ is covered by finitely many left translates of $\varphi(G)$ (see \cite[Lemma I.6.9]{PiGST}). However, the technology of generic types and sets can be employed outside of stable groups, most notably in expansions of groups with simple theories (where ``generic" and ``syndetic" need not coincide). 
\end{remark}

Another classical fact is that if $T$ is simple, then a formula $\varphi(x)$ is generic if and only if, for all $c\in \G$,  $\varphi(x\cdot c)$ does not fork over $G$ (see \cite[Proposition 3.10]{PiGDST}). The next result shows how ``$1$-sided nonforking substantial" is essentially obtained by replacing ``for all $c$" with ``there exists a $c$" (modulo certain caveats). Moreover, ``$1$-sided coheir substantial" is obtained from ``syndetic" in an analogous fashion.  In the following statement, we use $x\notin G$ as shorthand for the (partial) type $\{x\not=g : g\in G\}$.

\begin{proposition}\label{prop:generic}
Let $\varphi(x)$ be a formula.
\begin{enumerate}[$(a)$]
\item The following are equivalent.
\begin{enumerate}[$(i)$]
\item $\varphi(x)$ is $1$-sided coheir substantial.
\item There is $c\in \G\backslash G$ such that $\varphi(x\cdot c)$ has infinitely many solutions in $G$.
\end{enumerate}
\item The following are equivalent.
\begin{enumerate}[$(i)$]
\item $\varphi(x)$ is $1$-sided nonforking substantial.
\item There is $c\in \G\backslash G$ such that $\{\varphi(x\cdot c)\}\cup x\not\in G$ does not fork over $G$.
\end{enumerate}
\item The following are equivalent.
\begin{enumerate}[$(i)$]
\item $\varphi(x)$ is syndetic.
\item For all $c\in \G$, $\varphi(x\cdot c)$ has infinitely many solutions in $G$.
\item For all $c\in \G$, $\varphi(x\cdot c)$ has a solution in $G$.
\end{enumerate}
\item Assume $T$ is simple. The following are equivalent.
\begin{enumerate}[$(i)$]
\item $\varphi(x)$ is generic.
\item For all $c\in \G$, $\{\varphi(x\cdot c)\}\cup x\not\in G$ does not fork over $G$.
\item For all $c\in \G$, $\varphi(x\cdot c)$ does not fork over $G$.
\end{enumerate}
\end{enumerate}
\end{proposition}
\begin{proof}
Parts $(a)$ and $(b)$ are straightforward.

Part $(c)$. $(i)\Rightarrow (ii)$. Assume $\varphi(x)$ is syndetic. Fix $c\in \G$ and a finite subset $F\subseteq G$. We find $g\in G\backslash F$ such that $\varphi(g\cdot c)$ holds. By assumption, there are $a_1,\ldots,a_n\in G$ such that $\G=\bigcup_{i=1}^na_i\cdot\varphi(\G)$. Since $G$ is infinite, there is $b\in G$ such that $a^{\text{-}1}_i\cdot b^{\text{-}1}\not\in F$ for all $1\leq i\leq n$. Note that $\G=\bigcup_{i=1}^nb\cdot a_i\cdot\varphi(\G)$. So there is some $i\in\{1,\ldots,n\}$ such that $c\in b\cdot a_i\cdot\varphi(\G)$. Then $g:=a^{\text{-}1}_i\cdot b^{\text{-}1}$ is as desired.

$(ii)\Rightarrow (iii)$. Trivial.

$(iii)\Rightarrow (i)$. Suppose $\varphi(x)$ is not syndetic. Then the type $\{\neg\varphi(g\cdot y):g\in G\}$ is finitely satisfiable, and thus realized by  $c\in \G$. So $\varphi(x\cdot c)$ is not satisfied in $G$.

Part $(d)$. As noted above, the equivalence of $(i)$ and $(iii)$ is \cite[Proposition 3.10]{PiGDST}.

$(i)\Rightarrow(ii)$. Suppose $\varphi(x)$ is generic. Let $p\in S_1(G)$ be a generic type containing $\varphi(x)$. Fix $c\in \G$, and let $q(x)=p(x\cdot c)\in S_1(Gc)$. Since all formulae in $p$ are generic, it follows from $(i)\Rightarrow(iii)$ that $q$ does not fork over $G$. Let $X=p(\G)$, and note that $q(\G)=X\cdot c^{\text{-}1}$. Since $p$ is generic, it follows that $X$ is infinite. Therefore $q(\G)$ is infinite, and so there is a realization of $q$ in $\G\backslash G$. In particular, $\varphi(x\cdot c)\wedge x\not\in G$ is contained in $q$, and thus does not fork over $G$.

$(ii)\Rightarrow(iii)$. Trivial.
\end{proof}

\begin{remark}
It is worth emphasizing that, unlike the situation with generic formulae, in the characterization of $1$-sided nonforking substantial formulae, the type $\{\varphi(x\cdot c)\}\cup x\not\in G$ cannot be replaced by the formula $\varphi(x\cdot c)$. For example, let $G$ be the expansion of $(\Z,+,0)$ obtained by adding a predicate for $A=\{2^n:n\in\mathbb{N}\}$. Then $G$ is stable (see \cite{PaSk} or \cite{PoZ}). If $c\in A(\G)\backslash G$ is a nonstandard power of $2$, then the formula $A(x+c)$ is realized by $0\in G$ and thus does  not fork over $G$. But $A(x)$ is not nonforking substantial as the powers of $2$ do not have the sumset property.
\end{remark}

\begin{corollary}\label{cor:gensub}
Assume $T$ is simple. If a formula $\varphi(x)$ is generic then it is nonforking substantial.
\end{corollary}

In the case that $T$ is stable and $G$ is amenable, the implication given by the previous corollary can also be explained using Banach density and Corollary \ref{cor:BDnf}.

\begin{corollary}
Assume $T$ is stable and $G$ is amenable. Fix a definable set $A\subseteq G$.
\begin{enumerate}[$(a)$]
\item If $A$ is generic, then $\BD(A)>0$.
\item If $\BD(A)>0$, then $A$ is substantial.
\end{enumerate}
\end{corollary}
\begin{proof}
Part $(a)$ follows from known facts. In particular, it is a standard exercise that, for subsets of amenable groups, syndeticity corresponds to positive \emph{lower} Banach density (which implies positive upper Banach density). Moreover, we have already recalled that generic and syndetic coincide for definable subsets of stable groups.

Part $(b)$ follows from Corollary \ref{cor:BDnf} and Corollary \ref{cor:stable-simple}$(a)$.
\end{proof}

\begin{question}
Is there an amenable group $G$ and a set $A\subseteq G$, definable in a stable first-order expansion of $G$, such that $A$ is substantial and $\BD(A)=0$?
\end{question}

\begin{remark}
Given an arbitrary ternary relation $\ind$ defined on small subsets of $\G$, one can define a formula $\varphi(x)$ to be:
\begin{enumerate}[$(i)$]
\item \emph{$\ind$-generic} if for all $c\in\G\backslash G$ there is $b\in \G\backslash G$ such that $\varphi(b\cdot c)$ holds and $b\ind_G c$, and
\item \emph{1-sided $\ind$-substantial} if there are $b,c\in\G\backslash G$ such that $\varphi(b\cdot c)$ holds and $b\ind_G c$.
\end{enumerate}
For example, define $b\ind^u_G c$ to mean that $\tp(b/Gc)$ is finitely satisfiable in $G$. Then Theorem \ref{thm:1ch} says that the 1-sided productset property is equivalent to 1-sided $\ind^u$-substantial; and Proposition \ref{prop:generic}$(c)$ says that syndetic is equivalent to $\ind^u$-generic. If $\ind^f$ denotes nonforking independence, then Corollary \ref{cor:1psp-fork} says that the $1$-sided productset property implies $1$-sided $\ind^f$-substantial. In fact, this holds when $\ind^f$ is replaced by any ternary relation $\ind$ weaker than $\ind^u$. Another notable example is \thorn-independence, denoted $\ind^{\text{\thorn}}$, in $T^{\text{eq}}$. For real elements, $\ind^u$ is the same when evaluated in $T$ or in $T^{\text{eq}}$ and, moreover, implies $\ind^{\text{\thorn}}$.  If we further assume $T$ is rosy, then $\ind^{\text{\thorn}}$-generic, as defined above, agrees with the notion defined in \cite[Section 1]{EKP} (by following the proof of Proposition \ref{prop:generic} with forking  replaced by \thorn-forking, and \cite[Proposition 3.10]{PiGDST} replaced by \cite[Proposition 1.17]{EKP}).
\end{remark}

\subsection{Generically stable types and the finitary productset property}

Continuing with the analogy between Proposition \ref{prop:MDN} and Theorem \ref{thm:PSP}, it is clear that types that commute with other types with respect to the operation $\otimes$ will play a role in the investigation of sets with the productset property. Such types have been identified (in the NIP context) as the so-called \emph{generically stable types} and thus one should be motivated to consider a suitable notion of substantial formulae associated to such types.  The catch here is that generically stable types are defined with respect to some small model $M\prec \mathbb{G}$ which need not be equal to the original small model $G$.  Thus, one is led to the investigation of formulae $\varphi$ for which $\varphi(M)$ has the productset property.  In terms of the original group $G$, we are forced to consider the following weaker notion:

\begin{definition}
A formula $\varphi(x)$ has the \emph{finitary productset property} if, for any $n\in \N$, there are $B,C\subseteq G$ with $|B|,|C|\geq n$ such that $B\cdot C\subseteq \varphi(G)$.
\end{definition}

\begin{remark}
$\varphi(x)$ has the finitary productset property if and only if there is an elementary extension $G\prec M$ such that $\varphi(M)$ has the productset property.  Also, if $G$ is $\omega$-saturated, then $\varphi(x)$ has the finitary productset property if and only if it has the productset property.
\end{remark}

It is possible to carry out the above discussion precisely and prove that in NIP theories, formulae which are substantial with respect to generically stable types (as defined below) satisfy the finitary productset property.  However, by relativizing coheir substantiality to an arbitrary model $M$, we can obtain this result without any assumptions on $T$. In particular, given a small model $M\prec\G$, we say that a formula $\varphi(x)$ is \emph{coheir substantial over $M$} if it satisfies the definition of coheir substantiality (Definition \ref{def:cs}), with $G$ replaced by the arbitrary model $M$. We can now define a notion of substantial formulae motived by generically stable types. 

\begin{definition}
$~$
\begin{enumerate}
\item (From \cite{PiTa})  Given a small model $M\prec\G$, a type $p\in S_1(\G)$ is \emph{generically stable over $M$} if it is $M$-invariant and, for any Morley sequence $(b_i)_{i<\omega}$ in $p$ over $M$ and any formula $\theta(x)$ (with parameters from $\G$), the set $\{i\in\omega:\G\models\theta(b_i)\}$ is finite or cofinite.
\item A formula $\varphi(x)$ is \emph{gs-substantial} if there is a small model $M\prec\G$ and $b,c\in\G\backslash M$ such that $\varphi(b\cdot c)$ holds and $\tp(b/Gc)$ extends to a nonalgebraic global type which is generically stable over $M$.
\end{enumerate}
\end{definition}

The following fact on generically stable types (in arbitrary theories) is necessary for our analysis.

\begin{fact}\label{fact:PT}\cite[Proposition 1]{PiTa}
Fix a small model $M\prec\G$. If a global type $p$ is generically stable over $M$, then $p$ is definable over $M$ and finitely satisfiable in $M$.
\end{fact}

\begin{remark}
If $T$ is NIP, then the converse of Fact \ref{fact:PT} holds and provides one of the many equivalent ways to formulate the notion of generically stable global types (see, for example, \cite[Theorem 2.29]{simon}).  Moreover, \cite[Proposition 2.33]{simon} states that, in NIP theories, generically stable types commute (with respect to $\otimes$) with all invariant types. 
\end{remark}

\begin{theorem}\label{thm:gss}
If a formula $\varphi(x)$ is gs-substantial, then it has the finitary productset property. 
\end{theorem}
\begin{proof}
Assume that $\varphi(x)$ is gs-substantial as witnessed by $M\prec \G$ and $b,c\in \G\backslash M$.  By relativizing Theorem \ref{thm:ch}$(a)$, we see that if $\varphi(x)$ is coheir substantial over $M$, then $\varphi(M)$ has the productset property, whence $\varphi(x)$ has the finitary productset property.  Thus, it suffices to show that $\varphi(x)$ is coheir substantial over $M$. 

Let $p(x)$ be a nonalgebraic global extension of $\tp(b/Gc)$ which is generically stable over $M$. By Fact \ref{fact:PT}, $p$ is definable over $M$ and finitely satisfiable in $M$. Let $b'\models p|_{MGc}$. Then $b'\equiv_{Gc} b$ and so $b'\cdot c\in\varphi(\G)$. Since $p$ is nonalgebraic, we have $b'\not\in M$. Since $\tp(b'/Mc)$ is contained in $p(x)$, we have that $\tp(b'/Mc)$ is finitely satisfiable in $M$. We also have that $\tp(b'/Mc)$ is $M$-definable, and thus $\tp(c/Mb')$ is finitely satisfiable in $M$. Altogether $b',c\in \G\backslash M$ witness that $\varphi(x)$ is coheir substantial over $M$.
\end{proof}

From the proof of Theorem \ref{thm:gss} it is clear that if $\varphi(x)$ is gs-substantial as witnessed by $M=G$, then $\varphi(x)$ is coheir substantial and thus has the productset property. This result can, once again, be adapted to work outside of the setting of groups (see Remark \ref{rem:formula}).

\section{Erd\H{o}s's conjecture in groups with distal theories}\label{sec:distal}

Given a structure $M$, a \emph{Keisler measure on $M$} is a finitely additive probability measure on the definable subsets of $M$. Such a measure is \emph{generically stable} if it has a (unique) extension to a global $M$-invariant measure which is definable and finitely satisfiable in $M$ (see \cite[Section 7.5]{simon}). In the NIP context, the following equivalent definition can be used, which avoids reference to global extensions (see \cite[Theorem 7.29]{simon}). 

\begin{definition}\label{def:gsmeasure}
Let $M$ be a structure whose theory is NIP. A Keisler measure $\mu$ on $M$ is \emph{generically stable} if, for any formula $\varphi(x,\bar{z})$ and $\epsilon>0$, there are $a_1,\ldots,a_n\in M$ such that, for any $\bar{c}\in M^{|\bar{z}|}$, 
\[
\left|\mu(\varphi(M,\bar{c}))-\textstyle\frac{1}{n}|\{i:M\models\varphi(a_i,\bar{c})\}|\right|<\epsilon.
\]
\end{definition}

We also recall that if $I$ is an index set, $(M_i)_{i\in I}$ is a sequence of $\mathcal{L}$-structures, $\mu_i$ is Keisler measure on $M_i$, $\mathcal{U}$ is an ultrafilter on $I$, and $N:=\prod_{\mathcal{U}}M_i$, then one can define the ultralimit measure $\mu:=\prod_{\mathcal{U}}\mu_i$ on $N$ such that, for definable $D_i\seq M_i$, $\mu(\prod_{\mathcal{U}}D_i):=\lim_{\mathcal{U}}\mu_i(D_i)$. In the NIP setting, if each $\mu_i$ is generically stable on $M_i$, then $\mu$ is generically stable on $N$ (see \cite[Corollary 1.3]{Simon-note}). 

We now focus on \emph{distal structures}, which were first defined by P. Simon to capture the class of ``purely unstable NIP theories". (See \cite[Chapter 9]{simon} for a precise definition.)  Examples of distal structures include o-minimal structures, $(\Z,+,<,0)$, and, more generally, any ordered dp-minimal structure. In this section, we analyze productsets in distal expansions of groups. In light of Remark \ref{rem:formula}, the productset phenomenon is closely related to the recent study of regularity in distal theories. In particular, we cite the following fact.

\begin{fact}\textnormal{\cite[Corollary 4.6]{ChSt}}\label{fact:CS}
Let $M$ be a distal $\mathcal{L}$-structure and fix a formula $\varphi(x,y,\bar{z})\in\mathcal{L}$. Then there are $\epsilon>0$ and formulae $\theta_1(x,\bar{w}),\theta_2(y,\bar{w})\in\mathcal{L}$ such that, for any generically stable Keisler measures $\mu,\nu$ on $M$ and any $\bar{a}\in M^{|\bar{z}|}$, there are $\bar{c}_1,\bar{c}_2\in M^{|\bar{w}|}$ such that $\mu(\theta_1(M,\bar{c}_1)),\nu(\theta_2(M,\bar{c}_2))\geq\epsilon$ and either $\theta_1(M,\bar{c}_1)\times\theta_2(M,\bar{c}_2)\seq \varphi(M^2,\bar{a})$ or $\theta_1(M,\bar{c}_1)\times\theta_2(M,\bar{c}_2)\seq \neg\varphi(M^2,\bar{a})$.
\end{fact}

Although we have stated the previous result using singleton variables $x$ and $y$ (which will be sufficient for our purposes), it is worth noting that the results of \cite{ChSt} apply to formulae in any number of partitioned tuples of variables. In \cite[Theorem 5.11]{CGS}, Chernikov, Galvin, and Starchenko use Fact \ref{fact:CS} to prove essentially what we have stated below as Corollary \ref{cor:distalE}. Using similar methods, we prove the following refinement of \cite[Theorem 5.11]{CGS} and, for the sake of completeness, derive Corollary \ref{cor:distalE} directly from it. Once again, we have stated our results for formulae in two singleton variables $x$ and $y$, and the generalization to formulae in any number of partitioned tuples of variables is evident.

\begin{proposition}\label{prop:distal}
Let $M$ be a distal $\mathcal{L}$-structure and fix a formula $\varphi(x,y,\bar{z})\in\mathcal{L}$. Then there are formulae $\theta_1(x,\bar{w}),\theta_2(x,\bar{w})\in\mathcal{L}$ such that, for any $\bar{a}\in M^{|\bar{z}|}$, if for all $n>0$ there are $B,C\seq M$ such that $|B|,|C|\geq n$ and $B\times C\seq\varphi(M^2,\bar{a})$, then for all $n>0$ there are $\bar{c}_1,\bar{c}_2\in M^{|\bar{w}|}$ such that $|\theta_1(M,\bar{c}_1)|,|\theta_2(M,\bar{c}_2)|\geq n$ and $\theta_1(M,\bar{c}_1)\times\theta_2(M,\bar{c}_2)\seq\varphi(M^2,\bar{a})$.
\end{proposition}
\begin{proof}
Fix $\varphi(x,y,\bar{z})$. Fix also a nonprincipal ultrafilter $\mathcal{U}$ on $\N$, and set $N:=M^{\mathcal{U}}$. Let $\epsilon>0$ and $\theta_1(x,\bar{w}),\theta_2(y,\bar{w})$ be as in Fact \ref{fact:CS} with respect to the distal structure $N$. Fix $\bar{a}\in M^{|\bar{z}|}$ and assume that, for all $n\in\N$, there are $B_n,C_n\seq M$ such that $|B_n|=|C_n|=n$ and $B_n\times C_n\seq \varphi(M^2,\bar{a})$. Define, for $n\in\N$, the Keisler measures $\mu_n$ and $\nu_n$ on $M$, such that, given a definable set $D\seq M$, we have
\[
\mu_n(D)=\textstyle \frac{1}{n}|D\cap B_n|\makebox[.5in]{and}\nu_n(D)=\frac{1}{n}|D\cap C_n|.
\]
By Definition \ref{def:gsmeasure}, $\mu_n$ and $\nu_n$ are generically stable measures on $M$. Let $\mu=\prod_{\mathcal{U}}\mu_n$ and $\nu=\prod_{\mathcal{U}}\nu_n$. Then $\mu$ and $\nu$ are generically stable Keisler measures on $N$.  By Fact \ref{fact:CS}, there are $\bar{c}_1,\bar{c}_2\in N^{|\bar{w}|}$ such that, setting $B:=\theta_1(N,\bar{c}_1)$ and $C:=\theta_2(N,\bar{c}_2)$, we have $\mu(B),\nu(C)\geq\epsilon$ and either $B\times C\seq\varphi(N^2,\bar{a})$ or $B\times C\seq\neg\varphi(N^2,\bar{a})$. Since $N$ is an elementary extension of $M$, in order to finish the proof, it suffices to show that $B,C$ are infinite and $(B\times C)\cap \varphi(N^2,\bar{a})\neq\emptyset$. 

For $i\in\{1,2\}$ and $n\in\N$, let $\bar{c}_i^n\in M^{|\bar{w}|}$ be such that $\bar{c}_i=(\bar{c}_i^n)_{n\in\N}/\mathcal{U}$. Set $B^*_n:=\theta_1(M,\bar{c}^n_1)$ and $C^*_n:=\theta_2(M,\bar{c}^n_2)$. Then $B=\prod_{\mathcal{U}}B^*_n$ and $C=\prod_{\mathcal{U}}C^*_n$. Since $\mu(B)\geq\epsilon$, it follows that
\[
X:=\left\{n\in\N:|B^*_n\cap B_n|\geq\textstyle\frac{\epsilon}{2}|B_n|\right\}=\left\{n\in\N:\mu_n(B^*_n)\geq\textstyle\frac{\epsilon}{2}\right\}\in\mathcal{U}.
\]
Given $k\in\N$, there is a cofinite set $Y_k\seq\N$ such that $\frac{\epsilon}{2}|B_n|\geq k$ for all $n\in Y_k$. Thus, for any $k\in\N$, $X\cap Y_k\seq\{n\in N:|B^*_n|\geq k\}\in\mathcal{U}$, which implies $|B|\geq k$. Thus $B$ is infinite and, by a similar argument, $C$ is infinite. Also, if 
\[
Y:=\{n\in\N:(B^*_n\times C^*_n)\cap(B_n\times C_n)\neq\emptyset\},
\]
then $Y\in\mathcal{U}$ and $Y\seq\{n\in\N:(B^*_n\times C^*_n)\cap\varphi(M^2,\bar{a})\neq\emptyset\}$. Since $B\times C=\prod_{\mathcal{U}}(B^*_n\times C^*_n)$, it follows that $(B\times C)\cap\varphi(N^2,\bar{a})\neq\emptyset$.
\end{proof}

Recall that a structure $M$ \emph{eliminates $\exists^\infty$} if for any formula $\theta(x,\bar{w})$, there is some $n\in\N$ such that, for any $\bar{c}\in M^{|\bar{w}|}$, if $|\theta(M,\bar{c})|\geq n$, then $\theta(M,\bar{c})$ is infinite.

\begin{corollary}[\textnormal{Chernikov, Galvin, Starchenko \cite{CGS}}]\label{cor:distalE}
Let $M$ be a distal $\mathcal{L}$-structure with elimination of $\exists^\infty$.  Then for any formula $\varphi(x,y,\bar{z})\in\mathcal{L}$, there is some $n\in\N$ such that, for any $\bar{a}\in M^{|\bar{z}|}$, if there are $B',C'\seq M$ such that $|B'|,|C'|\geq n$ and $B'\times C'\seq\varphi(M^2,\bar{a})$, then there are infinite definable $B,C\seq M$ such that $B\times C\seq\varphi(M^2,\bar{a})$.
\end{corollary}
\begin{proof}
Suppose the corollary is false as witnessed by $\varphi(x,y,\bar{z})$. Then for any $n\in\N$, there is $\bar{a}_n\in M^{|\bar{z}|}$ such that $\varphi(M^2,\bar{a}_n)$ does not contain $B\times C$ for any infinite definable $B,C\seq M$, but there are $B_n,C_n\seq M$ such that $|B_n|,|C_n|\geq n$ and $B_n\times C_n\seq\varphi(M^2,\bar{a}_n)$. Let $\mathcal{U}$ be a nonprincipal ultrafilter on $\N$ and set $N:=M^{\mathcal{U}}$. Let $\theta_1(x,\bar{w})$ and $\theta_2(y,\bar{w})$ be given by Proposition \ref{prop:distal} (with respect to $\varphi(x,y,\bar{z})$ and the model $N$). Since $M$ eliminates $\exists^\infty$, there is some $m\in\N$ such that, for $i\in\{1,2\}$ and $\bar{c}\in M^{|\bar{w}|}$, if $|\theta_i(M,\bar{c})|\geq m$, then $\theta_i(M,\bar{c})$ is infinite. Let $\bar{a}^*=(\bar{a}_n)_{n\in\N}/\mathcal{U}\in N$. By assumption and \L{}o\'{s}'s Theorem, $\varphi(x,y,\bar{a})$ satisfies the conditions of Proposition \ref{prop:distal} (with respect to $N$). By choice of $\theta_1$ and $\theta_2$, it follows that $\bar{c}^*$ satisfies the formula $\psi(z)$ expressing ``there are $\bar{w}_1,\bar{w}_2$ such that $\theta_1(x,\bar{w}_1)$ and $\theta_2(y,\bar{w}_2)$ each have at least $m$ solutions and $\theta_1(x,\bar{w}_2)\times\theta_2(y,\bar{w}_2)$ is contained in $\varphi(x,y,\bar{z})$." Therefore there is some $n\in\N$ such that $M\models\psi(\bar{a}_n)$. By choice of $m$, this contradicts the assumption on $\varphi(x,y,\bar{a}_n)$.
\end{proof}

We now return to the setting of a fixed first-order expansion of a group $G$ (with the same notation and conventions described after Remark \ref{rem:MDN}).  

\begin{definition}$~$
\begin{enumerate}
\item A formula $\varphi(x)$ has the \emph{definable productset property} if there are infinite definable $B,C\seq G$ such that $B\cdot C\seq\varphi(G)$.
\item A formula $\varphi(x)$ has the \emph{definable finitary productset property} if there are formulae $\theta_1(x,\bar{w}),\theta_2(x,\bar{w})$ such that, for all $n\in\N$ there are $\bar{c}_1,\bar{c}_2\in G^{|\bar{w}|}$ such that $|\theta_1(G,\bar{c}_1)|,|\theta_2(G,\bar{c}_2|\geq n$ and $\theta_1(G,\bar{c}_2)\cdot\theta_2(G,\bar{c}_2)\seq\varphi(G)$.
\end{enumerate}
\end{definition}

\begin{theorem}\label{thm:distal}
Assume $G$ is distal, and fix a formula $\varphi(x)$.
\begin{enumerate}[$(a)$]
\item If $\varphi(x)$ has the finitary productset property, then it has the definable finitary productset property.
\item Suppose $G$ eliminates $\exists^\infty$. Then there is some $n\in\N$ such that, if there are $B,C\seq G$ such that $|B|,|C|\geq n$ and $B\cdot C\seq \varphi(G)$, then $\varphi(x)$ has the definable productset property.
\end{enumerate}
\end{theorem}
\begin{proof}
For part $(a)$, apply Proposition \ref{prop:distal} to the formula $\varphi(x\cdot y)$. For part $(b)$, apply Corollary \ref{cor:distalE} to the formula $\varphi(x\cdot y)$. 
\end{proof}

\begin{remark}
Although it will not be significant for the subsequent results, it is worth noting that in Theorem \ref{thm:distal}$(a)$, the formulae $\theta_1(x,\bar{w})$ and $\theta_2(x,\bar{w})$ which witness the definable finitary productset property do not depend on the parameters appearing $\varphi(x)$. Similarly, the integer $n$ in part $(b)$ does not depend on parameters.
\end{remark}

Theorem \ref{thm:distal} has the following consequences for expansions of amenable groups.


\begin{corollary}
Assume $G$ is a distal expansion of an amenable group.
\begin{enumerate}[$(a)$]
\item If $A\seq G$ is definable with $\BD(A)>0$, then $A$ has the definable finitary productset property.
\item Suppose $G$ eliminates $\exists^\infty$. If $A\seq G$ is definable with $\BD(A)>0$, then $A$ has the definable productset property.
\end{enumerate}
\end{corollary}
\begin{proof}
By Theorem \ref{thm:distal}, it suffices to show that sets of positive Banach density have the finitary productset property. To see this, we use Fact \ref{fact:DGJLLM} and the observation that the $1$-sided productset property implies the finitary productset property. Indeed, fix $A\seq G$ and suppose $(b_i)_{i<\omega},(c_i)_{i<\omega}$ are sequences of pairwise distinct elements such that $b_i\cdot c_j\in A$ for all $i\leq j$. For any $n\in\N$, if $B=\{b_0,\ldots,b_{n-1}\}$ and $C=\{c_{n-1},\ldots,c_{2n-2}\}$, then $B\cdot C\seq A$.
\end{proof}

The class of distal expansions of amenable groups includes any o-minimal expansion of an ordered group and, more generally, any dp-minimal expansion of an ordered group (in fact, such groups are always abelian \cite{SiDP}). In the o-minimal case, one also has elimination of $\exists^\infty$. It is also worth mentioning the fact that any $\aleph_0$-categorial dp-minimal group is nilpotent-by-finite \cite{LKS}, and thus amenable. Of course, in order to use the previous work to conclude the productset property for sets of positive Banach density, one would need the further assumption of stability or distality.

\section*{Acknowledgements}

The authors thank Renling Jin for suggesting the proof of Proposition \ref{positiveBD} and Mauro DiNasso for allowing us to use Proposition \ref{prop:MDN}. We also thank Sergei Starchenko for helpful conversations, and Pierre Simon for comments on an earlier draft, which improved some results in Section \ref{sec:ch} (especially Theorem \ref{thm:PSP}).

\bibliographystyle{amsplain}
\end{document}